\documentclass[11pt]{article}
\usepackage[T1]{fontenc}
\usepackage[utf8]{inputenc}

\usepackage{stmaryrd}
\usepackage{amssymb}
\usepackage{amsmath}
\usepackage{amsthm}
\allowdisplaybreaks
\usepackage{amscd}
\usepackage{graphicx}
\usepackage{hyperref}
\usepackage{geometry}
\usepackage[all]{xy}
\usepackage{enumerate}
\usepackage{bookmark}
\usepackage{amsfonts}
\usepackage{fancyhdr}
\usepackage[numbers,sort&compress]{natbib}
\usepackage{epsfig}
\usepackage{picins}
\usepackage{picinpar}
\usepackage{subfigure}
\usepackage{pstricks}
\usepackage{fancyvrb}
\usepackage{dsfont}
\usepackage{mathrsfs}
\usepackage{mathpazo}
\usepackage{avant}
\usepackage{pifont}

\def\vbar{\mathchoice{\vrule height6.3ptdepth-.5ptwidth.8pt\kern- .8pt}
{\vrule height6.3ptdepth-.5ptwidth.8pt\kern-.8pt} {\vrule
height4.1ptdepth-.35ptwidth.6pt\kern-.6pt} {\vrule
height3.1ptdepth-.25ptwidth.5pt\kern-.5pt}}

\newtheorem{theorem}{Theorem}[section]
\newtheorem{definition}[theorem]{Definition}
\newtheorem{lemma}[theorem]{Lemma}
\newtheorem{corollary}[theorem]{Corollary}
\newtheorem{remark}[theorem]{Remark}

\numberwithin{equation}{section}

\begin{document}
\title{Transposed Poisson structures on Schr\"{o}dinger algebra in $(n+1)$-dimensional space-time}
\author{
Yang Yang$^a$ , \ Xiaomin Tang$^{a,}$\footnote{Corresponding author, E-mail: x.m.tang@163.com},\ \ Abror Khudoyberdiyev$^b$
\\
{\small $^a$ School of Mathematical Science, Heilongjiang University, Harbin, 150080, P. R. China,}
\\ {\small $^b$ Institute of Mathematics Uzbekistan Academy of
Sciences, National University of Uzbekistan}
}
\date{}

\maketitle

\begin{abstract}
Transposed Poisson   structures on the Schr\"{o}dinger algebra in $(n+1)$-dimensional space-time of Schr\"{o}dinger Lie groups are described.
 It was proven that the Schr\"{o}dinger algebra $\mathcal{S}_{n}$ in case of $n\neq 2$ does not have non-trivial $\frac{1}{2}$-derivations  and as it follows it does not admit non-trivial transposed Poisson structures.  All $\frac{1}{2}$-derivations and transposed Poisson structures for the algebra $\mathcal{S}_{2}$ are obtained.
 Also, we proved that the Schr\"{o}dinger algebra $\mathcal{S}_{2}$ admits a non-trivial ${\rm Hom}$-Lie structure.

\end{abstract}

\textbf{Key words}: Schr\"{o}dinger algebra; $\frac 12$-derivation; transposed Poisson algebra.

\textbf{Mathematics Subject Classification}: 17A30, 17B40, 17B63.

\section{Introduction}\ \ \ \ \ \ \ \ \

Poisson algebras arose from the study of Poisson geometry in the 1970s and have appeared in an extremely wide range of areas in mathematics and physics, such as Poisson manifolds, algebraic geometry, operads, quantization theory, quantum groups, and classical and quantum mechanics. The study of Poisson algebras also led to other algebraic structures, such as noncommutative Poisson algebras \cite{x1}, generic Poisson algebras \cite{ksu,kms}, Poisson bialgebras\cite{lbs,nb}, algebras of Jordan brackets and generalized Poisson algebras \cite{ck1,ck2,k,ki,z}, Gerstenhaber algebras \cite{g}, $F$-manifold algebras \cite{d}, Novikov-Poisson algebras \cite{x2}, quasi-Poisson algebras \cite{b}, double Poisson algebras \cite{v}, Poisson $n$-Lie algebra \cite{ck3}, etc. Recently, a dual notion of the Poisson algebra(transposed Poisson algebra), by exchanging the roles of the two binary operations in the Leibniz rule defining the Poisson algebra, has been introduced in the paper \cite{bbgw} of Bai, Bai, Guo, and Wu. They have shown that the transposed Poisson algebra defined in this way not only shares common properties with the Poisson algebra, including the closure undertaking tensor products and the Koszul self-duality as operad, but also admits a rich class of identities. More recently, a relation between $\frac 12$-derivations of Lie algebras and transposed Poisson algebras has not only been established \cite{fml}, but also between $\frac12$-biderivations and transposed Poisson algebras \cite{yh}. These ideas were used for describing all transposed Poisson structures on the Witt algebra which was one of the first examples of non-trivial transposed Poisson algebras \cite{fml}, the Virasoro algebra \cite{fml}, the algebra $\mathcal W(a,b)$ \cite{fml}, the thin Lie algebra \cite{fml}, the twisted Heisenberg-Virasoro algebra \cite{yh}, the Schr\"{o}dinger-Virasoro algebra \cite{yh}, the extended Schr\"{o}dinger-Virasoro algebra \cite{yh}, the 3-dimensional Heisenberg Lie algebra \cite{yh}, Block Lie algebras and superalgebras \cite{kk2}, Witt type algebras \cite{kk3}, oscillator Lie algebras \cite{bfk} and Galilean and solvable Lie algebras \cite{kl}. A list of actual open questions on transposed Poisson algebras was given in \cite{bfk}.

The Schr\"{o}dinger Lie group describes symmetries of the free particle Schr\"{o}dinger equation, see \cite{per}. For any positive integer $n$, the Lie algebra $\mathcal{S}_{n}$ in $(n+1)$-dimensional space-time of the Schr\"{o}dinger Lie group is called the Schr\"{o}dinger algebra, see \cite{Dob1997}. The Schr\"{o}dinger algebra $\mathcal{S}_{n}$ is a non-semisimple Lie algebra and also plays an important role in theoretical physics. Recently there was a series of papers on studying the structure and representation theory of the Schr\"{o}dinger algebra $\mathcal{S}_{1}$ in $(1+1)$-dimensional space-time, see \cite{AD,wq1,Dob1997,ZC,ty,wq2}. Representations and (bi)derivations over $\mathcal{S}_{n}$ are studied in \cite{Liu,7}.

Now let us recall the definition of the Schr\"{o}dinger algebra $\mathcal{S}_{n}$ in $(n+1)$-dimensional space-time by \cite{Dob1997,Liu} as follows. Throughout this paper, we denote by $\mathbb{C}$ the set of all complex numbers. More precisely, we have
\begin{definition}\label{d1}
The Schr\"{o}dinger algebra $\mathcal{S}_{n}$ is a Lie algebra with a $\mathbb{C}$-basis
$$
\{e,f,h,z,x_{i},y_{i},s_{jk}(=-s_{kj})\mid1\leqslant i\leqslant n,1\leqslant j<k\leqslant n\}
$$
equipped with the following non-trivial commutation relations
\begin{eqnarray*}
&& [e,f]=h,\quad[h,e]=2e,\quad [f,h]=2f,\\
&& [x_{i}, y_{i}]=z, \quad [h,x_{i}]=x_{i}, \quad [h, y_{i}]=-y_{i} \\
&& [e,y_{i}]=x_{i}, \quad [f,x_{i}]=y_{i},\\
&&  [s_{jk},x_{i}]=\delta_{ki}x_{j}-\delta_{ji}x_{k}, \quad [s_{jk},y_{i}]=\delta_{ki}y_{j}-\delta_{ji}y_{k},\\
&& [s_{jk},s_{lm}]=\delta_{lk}s_{jm}+\delta_{jm}s_{kl}+\delta_{mk}s_{lj}+\delta_{lj}s_{mk},
\end{eqnarray*}
where $\delta$ is the Kronecker Delta defined as $1$ for $i=j$ and as $0$ otherwise.
\end{definition}

 The Schr{\"o}dinger algebra $\mathcal{S}_{n}$ is a finite-dimensional, non-semisimple and non-solvable Lie algebra, and it is the semidirect product Lie algebra
$$\mathcal{S}_{n}=(\mathfrak{sl}_{2}\oplus\mathfrak{so}_{n})\ltimes\mathfrak{h}_{n}$$
where $\mathfrak{sl}_{2}={\rm Span}_{\mathbb{C}}\{e,f,h\}$ is the 3-dimensional simple Lie algebra, $\mathfrak{so}_{n}={\rm Span}_{\mathbb{C}}\{s_{kl}\mid1\leqslant k<l\leqslant n\}$ is the orthogonal Lie algebra and $\mathfrak{h}_{n}={\rm Span}_{\mathbb{C}}\{z,x_{i},y_{i}\mid1\leqslant i\leqslant n\}$ is the Heisenberg Lie algebra, for details refer to \cite{Liu}.  In particular, since the Schr{\"o}dinger algebra $\mathcal{S}_{1}$ does not have any term as $s_{jk}$ and so that it is the following form.
\begin{remark}\label{rmk1}
The Schr{\"o}dinger algebra $\mathcal{S}_{1}$ is a Lie algebra with a $\mathbb{C}$-basis  $\{e,f,h,z,x_{1},y_{1}\}$
equipped with the following non-trivial commutation relations
\begin{eqnarray*}
&& [e,f]=h,\quad[h,e]=2e,\quad [f,h]=2f,\\
&& [x_{1}, y_{1}]=z, \quad [h,x_{1}]=x_{1}, \quad [h, y_{1}]=-y_{1}, \\
&& [e,y_{1}]=x_{1}, \quad [f,x_{1}]=y_{1}.
\end{eqnarray*}
\end{remark}
In addition, the Schr{\"o}dinger algebra $\mathcal{S}_{2}$ will be considered separately, and we highlight his structure separately as follows.
\begin{remark}\label{rmk2}
The Schr{\"o}dinger algebra $\mathcal{S}_{2}$ is a Lie algebra with a $\mathbb{C}$-basis
$
\{f, h, e, z, x_1, y_1, x_2, y_2, s_{12}\}
$
equipped with the following non-trivial commutation relations
\begin{eqnarray*}
&& [e,f]=h,\quad  [h,e]=2e, \quad [f,h]=2f, \\
&& [x_i, y_i] = z,\quad  [h, x_i] = x_i,\quad  [h, y_i] = -y_i, \\
&& [e, y_i] = x_i, \quad [f, x_i] = y_i, \\
&& [s_{12}, x_1] = -x_2, \quad [s_{12}, x_2] = x_1, \quad [s_{12}, y_1] = -y_2, \quad [s_{12}, y_2] = y_1.
\end{eqnarray*}
\end{remark}
The following lemma is useful and it can be directly verified.
\begin{lemma}\label{grad111}
The Schr{\"o}dinger algebra $\mathcal{S}_{n}$ is a finitely generated $\mathbb{Z}_{2}$-graded Lie algebra as follows
\begin{equation}\label{Grad}
\mathcal{S}_{n}=(\mathcal{S}_{n})_{\bar{0}}\bigoplus (\mathcal{S}_{n})_{\bar{1}},
\end{equation}
where
\begin{equation}\label{Grad1}
 (\mathcal{S}_{n})_{\bar{0}}={\rm Span}_{\mathbb{C}}\{e,f,h,z,s_{kl}\mid1\leqslant k<l\leqslant n\},\quad
(\mathcal{S}_{n})_{\bar{1}}={\rm Span}_{\mathbb{C}}\{x_{i},y_{i}\mid1\leqslant i\leqslant n\}.
\end{equation}
\end{lemma}

In the present paper, we will study the $\frac 12$-derivations of the Schr\"{o}dinger algebra $\mathcal{S}_{n}$ in $(n+1)$-dimensional space-time. We prove that every $\frac 12$-derivation on $\mathcal{S}_{n}$ is trivial for any integer $n\neq 2$ and it follows that there are no non-trivial transposed Poisson algebra structures defined on $\mathcal{S}_{n}$ $(n\neq2)$.
All $\frac{1}{2}$-derivations and transposed Poisson structures for the algebra $\mathcal{S}_{2}$ are obtained.
As an application, it is also proved that the Schr\"{o}dinger algebra $\mathcal{S}_{2}$ admits a non-trivial ${\rm Hom}$-Lie structure.

\section{Preliminaries}\ \ \ \ \ \ \ \ \

In this section, we recall some definitions and known results for studying transposed Poisson structures. Although all algebras and vector spaces are considered over complex field, many results can be proven over other fields without modifications of proofs.

\begin{definition}\label{d2}
Let $\mathcal {L}$ be a vector space equipped with two nonzero bilinear operations $\cdot$ and $[\cdot, \cdot]$. The triple $(\mathcal {L}, \cdot, [\cdot, \cdot])$ is called a transposed Poisson algebra if $(\mathcal {L}, \cdot)$ is a commutative associative algebra and $(\mathcal {L}, [\cdot, \cdot])$ is a Lie algebra that satisfies the following compatibility condition
$$
\begin{array}{c}
2z \cdot [x,y]=[z \cdot x, y]+[x, z \cdot y].
\end{array}
$$
\end{definition}

Transposed Poisson algebras were first introduced in a paper by Bi, Bai, Guo and Wu \cite{bbgw}.

\begin{definition}\label{d2222}
Let $(\mathcal {L}, [\cdot, \cdot])$ be a Lie algebra. A transposed Poisson structure on $(\mathcal {L}, [\cdot, \cdot])$ is a commutative associative multiplication $\cdot$ in $\mathcal {L}$ which makes $(\mathcal {L}, \cdot, [\cdot, \cdot])$ a transposed Poisson algebra.
\end{definition}

\begin{definition}\label{d3}
Let $(\mathcal {L}, [\cdot, \cdot])$ be a Lie algebra, $\varphi: \mathcal {L}\rightarrow \mathcal {L} $ be a linear map. Then $\varphi$ is a $\frac 12$-derivation if it satisfies
$$
\begin{array}{c}
\varphi ([x,y])= \frac 12 ([\varphi (x), y]+[x, \varphi (y)]).
\end{array}
$$
\end{definition}

Observe that $\frac 12$-derivations are a particular case of $\delta$-derivations intoduced by Filippov in 1998 \cite{fil1}.
The main examples of $\frac 12$-derivations is the multiplication by an element from the ground field. Let us call such $\frac 12$-derivations as \emph{trivial
$\frac 12$-derivations.}

Let $G$ be an abelian group, $\mathcal{L}=\bigoplus \limits_{g\in G}\mathcal{L}_{g}$ be a $G$-graded Lie algebra.
We say that a $\frac 12$-derivation $\varphi$ has degree $g$ (deg($\varphi$)=$g$) if $\varphi(\mathcal{L}_{h})\subseteq \mathcal{L}_{g+h}, \forall h\in G $.
Let $\triangle(\mathcal{L})$ denote the space of $\frac 12$-derivations and write $\triangle_{g}(\mathcal{L})=\{\varphi \in \triangle(\mathcal{L}) \mid \deg(\varphi)=g\}$.
The following lemma is useful in our proof, but we did not find a proof of it, so we will present it below.

\begin{lemma}\label{l01}
Let $\mathcal{L}=\bigoplus \limits_{g\in G}\mathcal{L}_{g}$ be a $G$-graded Lie algebra and finitely generated (as a Lie algebra). Then
$$
\triangle(\mathcal{L})=\bigoplus \limits_{g\in G}\triangle_{g}(\mathcal{L}).
$$
\end{lemma}
\begin{proof}
For each element $g\in G$, we let $\rho_{g}:\mathcal{L}\rightarrow \mathcal{L}_{g}$ denote the canonical projection.
According to the assumption, there is a finite subset $S\subseteq \mathcal{L}$
generating $\mathcal{L}$. Let $\varphi$ be a $\frac 12$-derivation. Then there are finite sets $Q, R \subseteq G$, such that
\begin{equation}\label{e01}
\begin{array}{c}
S\subseteq \bigoplus \limits_{g\in Q}\mathcal{L}_{g} \quad and \quad \varphi(S)\subseteq \bigoplus \limits_{g\in R}\mathcal{L}_{g}.
\end{array}
\end{equation}
For $g\in G$, put $\varphi_{g}=\sum\limits_{m\in G}\rho_{g+m} \circ \varphi \circ \rho_{m}$. Since for $x_{h}\in \mathcal{L}_{h}$ and $x_{k}\in \mathcal{L}_{k}$, we have
$$
\setlength{\arraycolsep}{1.5pt}
\begin{array}{rcl}
\varphi_{g}([x_{h},x_{k}])&=&\sum\limits_{m\in G}\rho_{g+m} \circ \varphi \circ \rho_{m}([x_{h},x_{k}])\\
&=&\rho_{g+h+k}\circ \varphi([x_{h},x_{k}])\\
&=&\frac 12(\rho_{g+h+k}([\varphi(x_{h}),x_{k}])+\rho_{g+h+k}([x_{h},\varphi (x_{k})]))\\
&=&\frac 12([\rho_{g+h}(\varphi(x_{h})),x_{k}]+[x_{h},\rho_{g+k}(\varphi(x_{k}))])\\
&=&\frac 12([\sum\limits_{m\in G}\rho_{g+m} \circ \varphi \circ \rho_{m}(x_{h}),x_{k}]+[x_{h},\sum\limits_{m\in G}\rho_{g+m} \circ \varphi \circ \rho_{m}(x_{k})])\\
&=&\frac 12([\varphi_{g}(x_{h}),x_{k}]+[x_{h},\varphi_{g}(x_{k})]).
\end{array}
$$
Suppose $\varphi(x_{k})=\sum\limits_{i\in G}\alpha_{i}$, where $\alpha_{i}\in \mathcal{L}_{i}$. Then
$$
\setlength{\arraycolsep}{1.5pt}
\begin{array}{rcl}
\varphi_{g}(x_{k})&=&\sum\limits_{m\in G}\rho_{g+m} \circ \varphi \circ \rho_{m}(x_{k})\\
&=&\rho_{g+k} \circ \varphi(x_{k})\\
&=&\rho_{g+k}(\sum\limits_{i\in G}\alpha_{i})\\
&=&\alpha_{g+k}\in \mathcal{L}_{g+k}.
\end{array}
$$
It follows that $\varphi_{g}$ is contained in $\triangle_{g}(\mathcal{L})$.

Let $T=\{g-h\mid h\in Q, g\in R\}$. Then $T$ is finite and we obtain, observing (\ref{e01}), for $y\in S$,
$$
\setlength{\arraycolsep}{1.5pt}
\begin{array}{rcl}
\varphi(y)&=&\sum\limits_{g\in R}\rho_{g}\circ \varphi(y)\\
&=&\sum\limits_{g\in R}\sum\limits_{h\in Q}\rho_{g} \circ \varphi \circ \rho_{h}(y)\\
&=&\sum\limits_{h\in Q}(\sum\limits_{g\in R}\rho_{(g-h)+h} \circ \varphi \circ \rho_{h}(y))\\
&=&\sum\limits_{h\in Q}(\sum\limits_{q\in T}\rho_{q+h} \circ \varphi \circ \rho_{h}(y))\\
&=&\sum\limits_{q\in T}\sum\limits_{h\in Q}\rho_{q+h} \circ \varphi \circ \rho_{h}(y)\\
&=&\sum\limits_{q\in T}\sum\limits_{h\in G}\rho_{q+h} \circ \varphi \circ \rho_{h}(y)\\
&=&\sum\limits_{q\in T}\varphi_{q}(y).
\end{array}
$$
This shows that the derivations $\varphi$ and $\sum\limits_{q\in T}\varphi_{q}$ coincide on $S$. As $S$ generates $\mathcal{L}$, we obtain
$\varphi=\sum\limits_{q\in T}\varphi_{q}$. This proves the lemma.
\end{proof}

On the other hand, it is easy to see that Definitions \ref{d1} and \ref{d2} imply the following key lemma.

\begin{lemma}\label{l1} (see {\rm\cite{fml}})
Let $(\mathcal {L}, \cdot, [\cdot, \cdot])$ be a transposed Poisson algebra and $z$ an arbitrary element from $\mathcal {L}$.
Then the left multiplication $l_{z}$ in the associative commutative algebra $(\mathcal {L}, \cdot)$ gives a $\frac 12$-derivation of the Lie algebra $(\mathcal {L}, [\cdot, \cdot])$.
\end{lemma}

By Lemma \ref{l1}, it is easy to prove the following lemma.

\begin{lemma}\label{t1} (see {\rm\cite{fml}})
Let $\mathcal {L}$ be a Lie algebra without non-trivial $\frac 12$-derivations. Then every transposed Poisson structure defined on $\mathcal {L}$ is trivial.
\end{lemma}

\section{Transposed Poisson algebra structures on the Schr\"{o}dinger algebra}\ \ \ \ \ \ \ \ \

In this section, we describe transposed Poisson algebra structures on the Schr\"{o}dinger algebra $\mathcal{S}_{n}.$
The Schr\"{o}dinger algebra $\mathcal{S}_{1}$ is a semidirect product of
the $3$-dimensional Lie algebra $\mathfrak{sl}_{2}$ and the $3$-dimensional Heisenberg algebra $\mathfrak{h}_{1}$. It is easy to see that all $\frac 1 2$-derivations of $\mathcal{S}_{1}$ are trivial (see also \cite[Corollary 9]{fml}).
For $n=2$, we now define a $\frac 1 2$-derivation of $\mathcal{S}_{2}$ as follows. The verification is straightforward.

\begin{definition}
The $\frac 1 2$-derivation $\Re$ of the Schr\"{o}dinger algebra $S_2$ is determined by
\begin{equation}\label{ders2}
\Re(s_{12}) =z, \ \ \Re (u)=0, \forall u\in \{f, h, e, z, x_1, y_1, x_2, y_2\}.
\end{equation}
\end{definition}

Now we study the transposed Poisson algebra structures on the Schr\"{o}dinger algebra $\mathcal{S}_{n}$. To obtain this result, we first have to prove a few lemmas.

\begin{lemma}\label{12}
Let $\varphi_0$ be a  $\frac 12$-derivation of $\mathcal{S}_{n}$ $(n \geq 3)$ such that
$$
\varphi_{0}((\mathcal{S}_{n})_{\bar{0}})\subseteq(\mathcal{S}_{n})_{\bar{0}}, \quad \varphi_{0}((\mathcal{S}_{n})_{\bar{1}})\subseteq(\mathcal{S}_{n})_{\bar{1}},
$$
where the graded spaces $(\mathcal{S}_{n})_{\bar{0}}$ and $(\mathcal{S}_{n})_{\bar{1}}$ are given by \eqref{Grad1}.
Then $\varphi_0$ is trivial.
\end{lemma}
\begin{proof}
By (\ref{Grad1}) we can assume that
\begin{eqnarray}
&&\varphi_{0}(x_{i})=\sum\limits_{j=1}^n\alpha_{ij}x_{j}+\sum\limits_{k=1}^n\beta_{ik}y_{k}, \label{xiform}\\ &&\varphi_{0}(y_{i})=\sum\limits_{j=1}^n\gamma_{ij}x_{j}+\sum\limits_{k=1}^n\epsilon_{ik}y_{k}. \label{yiform} \\
&&\varphi_{0}(e)=\rho_{1}e+\rho_{2}h+\rho_{3}f+\rho_{4}z+\sum\limits_{1\leqslant k<l\leqslant n}\rho_{kl}s_{kl},\label{eform}\\ &&\varphi_{0}(f)=\lambda_{1}e+\lambda_{2}h+\lambda_{3}f+\lambda_{4}z+\sum\limits_{1\leqslant k<l\leqslant n}\lambda_{kl}s_{kl},\label{fform}\\ &&\varphi_{0}(h)=\mu_{1}e+\mu_{2}h+\mu_{3}f+\mu_{4}z+\sum\limits_{1\leqslant k<l\leqslant n}\mu_{kl}s_{kl}, \label{hform}\\
&&\varphi_{0}(s_{kl})=\sigma_{1}^{kl}e+\sigma_{2}^{kl}h+\sigma_{3}^{kl}f+\sigma_{4}^{kl}z+\sum\limits_{1\leqslant i<j\leqslant n}\sigma_{ij}^{kl}s_{ij}, \label{sform}
\end{eqnarray}
where $\alpha_{ij}, \beta_{ij}, \gamma_{ij}, \epsilon_{ij}, \rho_{i}, \rho_{ij}, \lambda_i, \lambda_{ij}, \mu_i, \mu_{ij}, \sigma_{i}^{kl}, \sigma_{ij}^{kl}$ are
some complex numbers which contains only finitely many nonzero.

Let \begin{equation}\label{theta}
\theta=\frac 12({\alpha_{11}+\epsilon_{11}}).
\end{equation}
Next, the lemma will be proved by the following steps.

 {\bf Step 1.}   By (\ref{xiform}) and (\ref{yiform}) we have
$$
\begin{array}{rcl}
\varphi_{0}(z)&=&\varphi_{0}([x_{i},y_{i}])=\frac 12([\varphi_{0}(x_{i}),y_{i}]+[x_{i},\varphi_{0}(y_{i})])\\
&=&\frac12 ([\sum\limits_{j=1}^n\alpha_{ij}x_{j}+\sum\limits_{k=1}^n\beta_{ik}y_{k},y_{i}]+[x_{i},\sum\limits_{j=1}^n\gamma_{ij}x_{j}+\sum\limits_{k=1}^n\epsilon_{ik}y_{k}])\\
&=&\frac 12(\alpha_{ii}+\epsilon_{ii}) z.
\end{array}
$$
This means $\alpha_{ii}+\epsilon_{ii}=\alpha_{11}+\epsilon_{11}$ for all $i=1, \cdots, n$, and together with (\ref{theta}) we have
\begin{equation}\label{zform}
\varphi_{0}(z)=\theta z.
\end{equation}

 {\bf Step 2.} By (\ref{eform}) and (\ref{yiform}) we have
$$
\setlength{\arraycolsep}{1.5pt}
\begin{array}{rcl}
\varphi_{0}(x_{i})&=&\varphi_{0}([e,y_{i}])=\frac12 ([\varphi_{0}(e),y_{i}]+[e,\varphi_{0}(y_{i})])\\
&=&\frac12([\rho_{1}e+\rho_{2}h+\rho_{3}f+\rho_{4}z+\sum\limits_{1\leqslant k<l\leqslant n}\rho_{kl}s_{kl},y_{i}]+
[e,\sum\limits_{j=1}^n\gamma_{ij}x_{j}+\sum\limits_{k=1}^n\epsilon_{ik}y_{k}])\\
&=&\frac12 ((\rho_{1}+\epsilon_{ii})x_{i}-\rho_{2}y_{i}+\sum\limits_{\begin{subarray}{1}{k=1}\\k\ne i\end{subarray}}^{n}\epsilon_{ik}x_{k}+
\sum\limits_{1\leqslant k<i\leqslant n}\rho_{ki}y_{k}-\sum\limits_{1\leqslant i<l\leqslant n}\rho_{il}y_{l}).
\end{array}
$$
It follows by (\ref{xiform}) that
\begin{equation}\label{e2}
\begin{array}{c}
\rho_{1}=2\alpha_{ii}-\epsilon_{ii},
\end{array}
\end{equation}
\begin{equation}\label{e3}
\begin{array}{c}
\beta_{ii}=-\frac 12{\rho_{2}},
\end{array}
\end{equation}
\begin{equation}\label{e4}
\begin{array}{c}
\alpha_{ij}=\frac 12 {\epsilon_{ij}},\ (j\ne i),
\end{array}
\end{equation}
\begin{equation}\label{e5}
\begin{array}{c}
\beta_{ik}=\frac 12{\rho_{ki}},(k<i),\quad \beta_{ik}=-\frac 12 {\rho_{ik}},(k>i).
\end{array}
\end{equation}
Similarly, by (\ref{eform}) and (\ref{xiform}) we have
$$
\setlength{\arraycolsep}{1.5pt}
\begin{array}{rcl}
0&=&\varphi_{0}([e,x_{i}])=\frac12 ([\varphi_{0}(e),x_{i}]+[e,\varphi_{0}(x_{i})])\\
&=&\frac12 ([\rho_{1}e+\rho_{2}h+\rho_{3}f+\rho_{4}z+\sum\limits_{1\leqslant k<l\leqslant n}\rho_{kl}s_{kl},x_{i}]+
[e,\sum\limits_{j=1}^n\alpha_{ij}x_{j}+\sum\limits_{k=1}^n\beta_{ik}y_{k}])\\
&=&\frac12 ((\rho_{2}+\beta_{ii})x_{i}+\rho_{3}y_{i}+\sum\limits_{1\leqslant k<i\leqslant n}(\beta_{ik}+\rho_{ki})x_{k}+
\sum\limits_{1\leqslant i<k\leqslant n}(\beta_{ik}-\rho_{ik})x_{k}),
\end{array}
$$
which yields
\begin{equation}\label{e6}
\begin{array}{c}
\rho_{2}+\beta_{ii}=0,
\end{array}
\end{equation}
\begin{equation}\label{e7}
\begin{array}{c}
\rho_{3}=0,
\end{array}
\end{equation}
\begin{equation}\label{e8}
\begin{array}{c}
\beta_{ik}+\rho_{ki}=0\, (k<i), \quad\beta_{ik}-\rho_{ik}=0\, (k>i).
\end{array}
\end{equation}
Substituting(\ref{e3}) and (\ref{e5}) into (\ref{e6}) and (\ref{e8}) respectively, we have
\begin{equation}\label{e9}
\begin{array}{c}
\rho_{2}=0,\quad\rho_{ki}=0\,(k<i),\quad\rho_{ik}=0\,(k>i),
\end{array}
\end{equation}
and
\begin{equation}\label{e10}
\begin{array}{c}
\beta_{ik}=0.
\end{array}
\end{equation}
Now by (\ref{fform}) and (\ref{xiform}) we have
$$
\setlength{\arraycolsep}{1.5pt}
\begin{array}{rcl}
\varphi_{0}(y_{i})&=&\varphi_{0}([f,x_{i}])=\frac12 ([\varphi_{0}(f),x_{i}]+[f,\varphi_{0}(x_{i})])\\
&=&\frac12 ([\lambda_{1}e+\lambda_{2}h+\lambda_{3}f+\lambda_{4}z+\sum\limits_{1\leqslant k<l\leqslant n}\lambda_{kl}s_{kl},x_{i}]+
[f,\sum\limits_{j=1}^n\alpha_{ij}x_{j}+\sum\limits_{k=1}^n\beta_{ik}y_{k}])\\
&=&\frac12(\lambda_{2}x_{i}+(\lambda_{3}+\alpha_{ii})y_{i}+\sum\limits_{1\leqslant k<i\leqslant n}\lambda_{ki}x_{k}-
\sum\limits_{1\leqslant i<l\leqslant n}\lambda_{il}x_{l}+\sum\limits_{\begin{subarray}{1}{j=1}\\j\ne i\end{subarray}}^{n}\alpha_{ij}y_{j}),
\end{array}
$$
which together with (\ref{yiform}) gives
\begin{equation}\label{e11}
\begin{array}{c}
\gamma_{ii}=\frac 12{\lambda_{2}},
\end{array}
\end{equation}
\begin{equation}\label{e12}
\begin{array}{c}
\lambda_{3}=2\epsilon_{ii}-\alpha_{ii},
\end{array}
\end{equation}
\begin{equation}\label{e13}
\begin{array}{c}
\gamma_{ij}=\frac 12{\lambda_{ji}}\,(j<i),\quad\gamma_{ij}=-\frac 12 {\lambda_{ij}}\,(j>i),
\end{array}
\end{equation}
\begin{equation}\label{e14}
\begin{array}{c}
\epsilon_{ij}=\frac 12{\alpha_{ij}}\,(j\ne i).
\end{array}
\end{equation}
Substituting (\ref{e4}) into (\ref{e14}), we have
\begin{equation}\label{e15}
\begin{array}{c}
\alpha_{ij}=0\,(j\ne i),
\end{array}
\end{equation}
\begin{equation}\label{e16}
\begin{array}{c}
\epsilon_{ij}=0\,(j\ne i).
\end{array}
\end{equation}
Similarly, by (\ref{fform}) and (\ref{yiform}) we have
$$
\setlength{\arraycolsep}{1.5pt}
\begin{array}{rcl}
0&=&\varphi_{0}([f,y_{i}])=\frac12 ([\varphi_{0}(f),y_{i}]+[f,\varphi_{0}(y_{i})])\\
&=&\frac12 ([\lambda_{1}e+\lambda_{2}h+\lambda_{3}f+\lambda_{4}z+\sum\limits_{1\leqslant k<l\leqslant n}\lambda_{kl}s_{kl},y_{i}]+
[f,\sum\limits_{j=1}^n\gamma_{ij}x_{j}+\sum\limits_{k=1}^n\epsilon_{ik}y_{k}])\\
&=&\frac 12(\lambda_{1}x_{i}+(\gamma_{ii}-\lambda_{2})y_{i}+\sum\limits_{1\leqslant k<i\leqslant n}(\gamma_{ik}+\lambda_{ki})y_{k}+
\sum\limits_{1\leqslant i<k\leqslant n}(\gamma_{ik}-\lambda_{ik})y_{k}),
\end{array}
$$
which gives
$$
\begin{array}{c}
\lambda_{1}=0,
\end{array}
$$
\begin{equation}\label{e18}
\begin{array}{c}
\lambda_{2}=\gamma_{ii},
\end{array}
\end{equation}
\begin{equation}\label{e19}
\begin{array}{c}
\gamma_{ij}+\lambda_{ji}=0\,(j<i),\quad \gamma_{ij}-\lambda_{ij}=0\, (j>i).
\end{array}
\end{equation}
Substituting(\ref{e11}) and (\ref{e13}) into (\ref{e18}) and (\ref{e19}) respectively, we have
$$
\begin{array}{c}
\lambda_{2}=0,\quad \lambda_{ji}=0\,(j<i),\quad\lambda_{ij}=0\, (j>i),
\end{array}
$$
\begin{equation}\label{e21}
\begin{array}{c}
\gamma_{ij}=0.
\end{array}
\end{equation}
It follows by (\ref{e10}), (\ref{e15}) that
\begin{equation}\label{xiform1}
\varphi_{0}(x_{i})=\alpha_{ii}x_{i}.
\end{equation}
and by (\ref{e16}), (\ref{e21}) that
\begin{equation}\label{yiform1}
\varphi_{0}(y_{i})=\epsilon_{ii}y_{i}.
\end{equation}

 {\bf Step 3.} By (\ref{hform}) and (\ref{xiform1}) we have
$$
\begin{array}{rcl}
\varphi_{0}(x_{i})&=&\varphi_{0}([h,x_{i}])=\frac12 ([\varphi_{0}(h),x_{i}]+[h,\varphi_{0}(x_{i})])\\
&=&\frac12([\mu_{1}e+\mu_{2}h+\mu_{3}f+\mu_{4}z+\sum\limits_{1\leqslant k<l\leqslant n}\mu_{kl}s_{kl},x_{i}]+[h,\alpha_{ii}x_{i}])\\
&=&\frac12((\mu_{2}+\alpha_{ii})x_{i}+\mu_{3}y_{i}+\sum\limits_{1\leqslant k<i\leqslant n}\mu_{ki}x_{k}-\sum\limits_{1\leqslant i<l\leqslant n}\mu_{il}x_{l}),
\end{array}
$$
which together with (\ref{xiform1}) gives
\begin{equation}\label{e22}
\mu_{2}=\alpha_{ii},
\end{equation}
$$
\mu_{3}=0, \quad\mu_{ki}=0\,(k<i), \quad \mu_{ik}=0\,(k>i).
$$
Similarly, By (\ref{hform}) and (\ref{yiform1}) we obtain
$$
\setlength{\arraycolsep}{1.5pt}
\begin{array}{rcl}
\varphi_{0}(y_{i})&=&\varphi_{0}([y_{i},h])=\frac12 ([\varphi_{0}(y_{i}),h]+[y_{i},\varphi_{0}(h)])\\
&=&\frac12([\epsilon_{ii}y_{i},h]+[y_{i},\mu_{1}e+\mu_{2}h+\mu_{3}f+\mu_{4}z+\sum\limits_{1\leqslant k<l\leqslant n}\mu_{kl}s_{kl}])\\
&=&\frac12(-\mu_{1}x_{i}+(\epsilon_{ii}+\mu_{2})y_{i}-\sum\limits_{1\leqslant k<i\leqslant n}\mu_{ki}y_{k}+\sum\limits_{1\leqslant i<l\leqslant n}\mu_{il}y_{l}),
\end{array}
$$
which together with (\ref{yiform1}) implies
$$
\mu_{1}=0,
$$
\begin{equation}\label{e25}
\mu_{2}=\epsilon_{ii}.
\end{equation}
It follows by (\ref{e22}), (\ref{e25}), (\ref{theta}), (\ref{e2}) and (\ref{e12}) that
\begin{equation}\label{e26}
\epsilon_{ii}=\mu_{2}=\alpha_{ii}=\theta=\rho_{1}=\lambda_{3}.
\end{equation}
Furthermore, by (\ref{xiform1}) and (\ref{yiform1})  with (\ref{e26}) one has
\begin{equation}\label{xiyiform2}
\varphi_{0}(x_{i})=\theta x_{i}, \quad\varphi_{0}(y_{i})=\theta y_{i}.
\end{equation}

 {\bf Step 4.}
By (\ref{e7}), (\ref{e9}), (\ref{e26}) with (\ref{eform}) we obtain
$$
\varphi_{0}(e)=\theta e+\rho_{4}z.
$$
Similarly, we have
\begin{equation}\label{fhform1}
\varphi_{0}(f)=\theta f+\lambda_{4}z,\quad \varphi_{0}(h)=\theta h+\mu_{4}z.
\end{equation}
Then we have
$$
\setlength{\arraycolsep}{1.5pt}
\begin{array}{rcl}
2\varphi_{0}(e)&=&\varphi_{0}([h,e])=\frac12 ([\varphi_{0}(h),e]+[h,\varphi_{0}(e)])\\
&=&\frac12([\theta h+\mu_{4}z,e]+[h,\theta e+\rho_{4}z])\\
&=&2\theta e,
\end{array}
$$
which gives
\begin{equation}\label{eform2}
\varphi_{0}(e)= \theta e.
\end{equation}
Similarly, by $-2f=[h,f]$ and (\ref{fhform1}) we obtain
\begin{equation}\label{fform2}
\varphi_{0}(f)= \theta f.
\end{equation}
Now by (\ref{eform2}) and (\ref{fform2}) one has
\begin{equation}\label{hform2}
\setlength{\arraycolsep}{1.5pt}
\begin{array}{rcl}
\varphi_{0}(h)&=&\varphi_{0}([e,f])=\frac12 ([\varphi_{0}(e),f]+[e,\varphi_{0}(f)])\\
&=&\frac12([\theta e,f]+[e,\theta f])\\
&=&\theta h.
\end{array}
\end{equation}

 {\bf Step 5.} By (\ref{sform}) and (\ref{xiyiform2}) we have
$$
\setlength{\arraycolsep}{1.5pt}
\begin{array}{rcl}
\varphi_{0}(x_{k})&=&\varphi_{0}([s_{ki},x_{i}])=\frac12 ([\varphi_{0}(s_{ki}),x_{i}]+[s_{ki},\varphi_{0}(x_{i})])\\
&=&\frac12([\sigma_{1}^{ki}e+\sigma_{2}^{ki}h+\sigma_{3}^{ki}f+\sigma_{4}^{ki}z+\sum\limits_{1\leqslant l<j\leqslant n}\sigma_{lj}^{ki}s_{lj},x_{i}]
+[s_{ki},\theta x_{i}])\\
&=&\frac12((\sigma_{ki}^{ki}+\theta )x_{k}+\sigma_{2}^{ki}x_{i}+\sigma_{3}^{ki}y_{i}+
\sum\limits_{\begin{subarray}{c}{1\leqslant l<i\leqslant n}\\{l\ne k}\end{subarray}}\sigma_{li}^{ki}x_{l}-
\sum\limits_{1\leqslant i<j\leqslant n}\sigma_{ij}^{ki}x_{j}),
\end{array}
$$
which together with (\ref{xiyiform2}) gives
\begin{equation}\label{e27}
\theta =\sigma_{ki}^{ki},
\end{equation}
\begin{equation}\label{e28}
\sigma_{2}^{ki}=\sigma_{3}^{ki}=0\,(k<i),
\end{equation}
\begin{equation}\label{e29}
\sigma_{li}^{ki}=0\,(k<i, k\ne l<i), \quad \sigma_{il}^{ki}=0\,(k<i<l).
\end{equation}
Similarly, by (\ref{sform}) and (\ref{xiyiform2}) with $\varphi_{0}(y_{k})=\varphi_{0}([s_{ki},y_{i}])$
 gives
\begin{equation}\label{e30}
\sigma_{1}^{ki}=0\,(k<i).
\end{equation}
By (\ref{e28}), (\ref{e30}) and (\ref{sform}) we obtain
\begin{equation}\label{e31}
\begin{array}{c}
\varphi_{0}(s_{kl})=\sigma_{4}^{kl}z+\sum\limits_{1\leqslant i<j\leqslant n}\sigma_{ij}^{kl}s_{ij}.
\end{array}
\end{equation}
For different $i,k,l$ with $k<l$, by (\ref{e31}) and (\ref{xiyiform2}) we have
$$
\setlength{\arraycolsep}{1.5pt}
\begin{array}{rcl}
0&=&\varphi_{0}([s_{kl},x_{i}])=\frac 12([\varphi_{0}(s_{kl}),x_{i}]+[s_{kl},\varphi_{0}(x_{i})])\\
&=&\frac 12([\sigma_{4}^{kl}z+\sum\limits_{1\leqslant i<j\leqslant n}\sigma_{ij}^{kl}s_{ij},x_{i}]+[s_{kl},\theta x_{i}])\\
&=&\frac12 (\sum\limits_{1\leqslant j<i\leqslant n}\sigma_{ji}^{kl}x_{j}-\sum\limits_{1\leqslant i<t\leqslant n}\sigma_{it}^{kl}x_{t}),
\end{array}
$$
which yields
\begin{equation}\label{e32}
\sigma_{ji}^{kl}=0\,(j<i, i\ne k\, \text{and}\, i\ne l),\quad \sigma_{ij}^{kl}=0\,(j>i, i\ne k\, \text{and}\, i\ne l).
\end{equation}
It follows by (\ref{e31}, (\ref{e27}, (\ref{e29}) and (\ref{e32}) that
\begin{equation}\label{sform1}
\varphi_{0}(s_{kl})=\sigma_{4}^{kl}z+\theta s_{kl}.
\end{equation}
Then by (\ref{sform1}) one can get
\begin{equation}\label{sform2}
\setlength{\arraycolsep}{1.5pt}
\begin{array}{rcl}
\varphi_{0}(s_{kj})&=&\varphi_{0}([s_{ij},s_{ik}])=\frac 12([\varphi_{0}(s_{ij}),s_{ik}]+[s_{ij},\varphi_{0}(s_{ik})])\\
&=&\frac 12([\sigma_{4}^{ij}z+\theta s_{ij},s_{ik}]+[s_{ij},\sigma_{4}^{ik}z+\theta s_{ik}])\\
&=&\theta s_{kj}.
\end{array}
\end{equation}

Summarizing the results of (\ref{zform}), (\ref{xiyiform2}),  (\ref{eform2}), (\ref{fform2}), (\ref{hform2})and (\ref{sform2}), we see that
$\varphi_{0}(u)=\theta u$ for all $u\in \mathcal{S}_{n}$.
Thus, $\varphi_{0}$ is a trivial $\frac 12$-derivation. The proof is completed.
\end{proof}

\begin{lemma}\label{13}
Let $\varphi_1$ be a $\frac 12$-derivation of $\mathcal{S}_{n}$ $(n \geq 3)$ such that
$$
\varphi_{1}((\mathcal{S}_{n})_{\bar{0}})\subseteq(\mathcal{S}_{n})_{\bar{1}}, \quad \varphi_{1}((\mathcal{S}_{n})_{\bar{1}})\subseteq(\mathcal{S}_{n})_{\bar{0}}.
$$
Then we have $\varphi_1=0$.
\end{lemma}

\begin{proof}
By (\ref{Grad1}) we can assume that
\begin{eqnarray}
&&\varphi_{1}(e)=\sum\limits_{i=1}^n\rho_{i}x_{i}+\sum\limits_{j=1}^n\mu_{j}y_{j},\label{eeform}\\
&&\varphi_{1}(f)=\sum\limits_{i=1}^n\lambda_{i}x_{i}+\sum\limits_{j=1}^n\theta_{j}y_{j},\label{ffform}\\ &&\varphi_{1}(h)=\sum\limits_{i=1}^n\gamma_{i}x_{i}+\sum\limits_{j=1}^n\epsilon_{j}y_{j}, \label{hhform}\\
&&\varphi_{1}(s_{kl})=\sum\limits_{i=1}^n\omega_{i}^{kl}x_{i}+\sum\limits_{j=1}^n\nu_{j}^{kl}y_{j}, \label{ssform}\\
&&\varphi_{1}(x_{i})=\alpha_{1}^ie+\alpha_{2}^ih+\alpha_{3}^if+\alpha_{4}^iz+\sum\limits_{1\leqslant k<l\leqslant n}\alpha_{kl}^is_{kl}, \label{xixiform}\\
&&\varphi_{1}(y_{i})=\beta_{1}^ie+\beta_{2}^ih+\beta_{3}^if+\beta_{4}^iz+\sum\limits_{1\leqslant k<l\leqslant n}\beta_{kl}^is_{kl}, \label{yiyiform}
\end{eqnarray}
where $\rho_{i}, \mu_i, \lambda_i, \theta{i}, \gamma_{i}, \epsilon_{i}, \omega_{i}^{kl}, \nu_{i}^{kl}, \alpha_{j}^i, \alpha_{kl}^i, \beta_{j}^i, \beta_{kl}^i$
belong to $\mathbb{C}$ which contains only finitely many nonzero complex numbers.
Next, the lemma will be proved by the following steps.

 {\bf Step 1.}   By (\ref{hhform}) and (\ref{eeform}) we have
$$
\setlength{\arraycolsep}{1.5pt}
\begin{array}{rcl}
2\varphi_{1}(e)&=&\varphi_{1}([h,e])=\frac 12([\varphi_{1}(h),e]+[h,\varphi_{1}(e)])\\
&=&\frac 12([\sum\limits_{i=1}^n\gamma_{i}x_{i}+\sum\limits_{j=1}^n\epsilon_{j}y_{j},e]+[h,\sum\limits_{i=1}^n\rho_{i}x_{i}+\sum\limits_{j=1}^n\mu_{j}y_{j}])\\
&=&\frac12 (\sum\limits_{i=1}^n(\rho_{i}-\epsilon_{i})x_{i}-\sum\limits_{j=1}^n\mu_{j}y_{j}),
\end{array}
$$
which together with (\ref{eeform}) gives
\begin{equation}\label{e33}
3\rho_{i}+\epsilon_{i}=0\,(1\leqslant i \leqslant n),
\end{equation}
\begin{equation}\label{e34}
\mu_{j}=0\,(1\leqslant j \leqslant n).
\end{equation}
Similarly, by (\ref{hhform}) and (\ref{ffform}) we obtain
$$
\setlength{\arraycolsep}{1.5pt}
\begin{array}{rcl}
-2\varphi_{1}(f)&=&\varphi_{1}([h,f])=\frac 12([\varphi_{1}(h),f]+[h,\varphi_{1}(f)])\\
&=&\frac 12([\sum\limits_{i=1}^n\gamma_{i}x_{i}+\sum\limits_{j=1}^n\epsilon_{j}y_{j},f]+[h,\sum\limits_{i=1}^n\lambda_{i}x_{i}+\sum\limits_{j=1}^n\theta_{j}y_{j}])\\
&=&\frac12 (\sum\limits_{i=1}^n\lambda_{i}x_{i}-\sum\limits_{j=1}^n(\theta_{j}+\gamma_{j})y_{j}),
\end{array}
$$
which together with (\ref{ffform}) implies
\begin{equation}\label{e35}
\lambda_{i}=0\,(1\leqslant i \leqslant n),
\end{equation}
\begin{equation}\label{e36}
3\theta_{j}=\gamma_{j}\,(1\leqslant j \leqslant n).
\end{equation}
Now by (\ref{eeform}) and (\ref{ffform}) we have
$$
\setlength{\arraycolsep}{1.5pt}
\begin{array}{rcl}
\varphi_{1}(h)&=&\varphi_{1}([e,f])=\frac 12([\varphi_{1}(e),f]+[e,\varphi_{1}(f)])\\
&=&\frac 12([\sum\limits_{i=1}^n\rho_{i}x_{i}+\sum\limits_{j=1}^n\mu_{j}y_{j},f]+[e,\sum\limits_{i=1}^n\lambda_{i}x_{i}+\sum\limits_{j=1}^n\theta_{j}y_{j}])\\
&=&\frac12 (\sum\limits_{j=1}^n\theta_{j}x_{j}-\sum\limits_{i=1}^n\rho_{i}y_{i}),
\end{array}
$$
and furthermore by (\ref{hhform}) one has
\begin{equation}\label{e37}
\theta_{j}=2\gamma_{j}\,(1\leqslant j \leqslant n),
\end{equation}
\begin{equation}\label{e38}
\rho_{i}+2\epsilon_{i}=0\,(1\leqslant i \leqslant n).
\end{equation}
Substituting (\ref{e36}) and (\ref{e33}) into (\ref{e37}) and (\ref{e38}) respectively, we have
\begin{equation}\label{e39}
\theta_{j}=\gamma_{j}=0\,(1\leqslant j \leqslant n),
\end{equation}
\begin{equation}\label{e40}
\rho_{i}=\epsilon_{i}=0\,(1\leqslant i \leqslant n).
\end{equation}
It follows by (\ref{e40}), (\ref{e34}) that
\begin{equation}\label{eeform1}
\varphi_{1}(e)=0,
\end{equation}
by (\ref{e35}), (\ref{e39}) that
\begin{equation}\label{ffform1}
\varphi_{1}(f)=0,
\end{equation}
and by (\ref{e39}), (\ref{e40}) that
\begin{equation}\label{hhform1}
\varphi_{1}(h)=0.
\end{equation}

 {\bf Step 2.}   By (\ref{hhform1}) and (\ref{ssform}) we have
$$
\begin{array}{rcl}
0&=&\varphi_{1}([h,s_{kl}])=\frac 12[h,\varphi_{1}(s_{kl})]\\
&=&\frac 12[h,\sum\limits_{i=1}^n\omega_{i}^{kl}x_{i}+\sum\limits_{j=1}^n\upsilon_{j}^{kl}y_{j}]\\
&=&\frac 12(\sum\limits_{i=1}^n\omega_{i}^{kl}x_{i}-\sum\limits_{j=1}^n\upsilon_{j}^{kl}y_{j}),
\end{array}
$$
which gives
\begin{equation}\label{ssform1}
\begin{array}{c}
\varphi_{1}(s_{kl})=0.
\end{array}
\end{equation}

{\bf Step 3.} By (\ref{hhform1}) and (\ref{xixiform}) we have
$$
\setlength{\arraycolsep}{1.5pt}
\begin{array}{rcl}
\varphi_{1}(x_{i})&=&\varphi_{1}([h,x_{i}])=\frac 12[h,\varphi_{1}(x_{i})]\\
&=&\frac 12[h,\alpha_{1}^{i}e+\alpha_{2}^{i}h+\alpha_{3}^{i}f+\alpha_{4}^{i}z+\sum\limits_{1\leqslant k<l \leqslant n}\alpha_{kl}^{i}s_{kl}]\\
&=&\alpha_{1}^{i}e-\alpha_{3}^{i}f,
\end{array}
$$
which together with (\ref{xixiform}) implies
\begin{equation}\label{e41}
\begin{array}{c}
\alpha_{2}^{i}=\alpha_{3}^{i}=\alpha_{4}^{i}=\alpha_{kl}^{i}=0\,(k<l).
\end{array}
\end{equation}
Similarly, by (\ref{hhform1}) and (\ref{yiyiform}) we obtain
$$
\setlength{\arraycolsep}{1.5pt}
\begin{array}{rcl}
-\varphi_{1}(y_{i})&=&\varphi_{1}([h,y_{i}])=\frac 12[h,\varphi_{1}(y_{i})]\\
&=&\frac 12[h,\beta_{1}^{i}e+\beta_{2}^{i}h+\beta_{3}^{i}f+\beta_{4}^{i}z+\sum\limits_{1\leqslant k<l \leqslant n}\beta_{kl}^{i}s_{kl}]\\
&=&\beta_{1}^{i}e-\beta_{3}^{i}f,
\end{array}
$$
which together with (\ref{yiyiform}) gives
\begin{equation}\label{e42}
\begin{array}{c}
\beta_{1}^{i}=\beta_{2}^{i}=\beta_{4}^{i}=\beta_{kl}^{i}=0\,(k<l).
\end{array}
\end{equation}
Furthermore by (\ref{eeform1}) and (\ref{yiyiform}) one has
$$
\setlength{\arraycolsep}{1.5pt}
\begin{array}{rcl}
\varphi_{1}(x_{i})&=&\varphi_{1}([e,y_{i}])=\frac 12[e,\varphi_{1}(y_{i})]\\
&=&\frac 12[e,\beta_{1}^{i}e+\beta_{2}^{i}h+\beta_{3}^{i}f+\beta_{4}^{i}z+\sum\limits_{1\leqslant k<l \leqslant n}\beta_{kl}^{i}s_{kl}]\\
&=&\frac 12(-2\beta_{2}^{i}e+\beta_{3}^{i}h),
\end{array}
$$
and further by (\ref{xixiform}) one can get
\begin{equation}\label{e43}
\alpha_{1}^{i}=-\beta_{2}^{i},
\end{equation}
\begin{equation}\label{e44}
2\alpha_{2}^{i}=\beta_{3}^{i}.
\end{equation}
It follows by (\ref{e41}), (\ref{e42}), (\ref{e43}) that
\begin{equation}\label{xixiform1}
\begin{array}{c}
\varphi_{1}(x_{i})=0,
\end{array}
\end{equation}
and by (\ref{e41}), (\ref{e42}), (\ref{e44}) that
\begin{equation}\label{yiyiform1}
\varphi_{1}(y_{i})=0.
\end{equation}
Then using (\ref{xixiform1}) and (\ref{yiyiform1}) we get
\begin{equation}\label{zzform}
\varphi_{1}(z)=\varphi_{1}([x_{i},y_{i}])=\frac12 ([\varphi_{1}(x_{i}),y_{i}]+[x_{i},\varphi_{1}(y_{i})])=0.
\end{equation}

Summarizing the results of (\ref{eeform1}), (\ref{ffform1}),  (\ref{hhform1}), (\ref{ssform1}), (\ref{xixiform1}), (\ref{yiyiform1}) and (\ref{zzform}), we see that
$\varphi_{1}=0$. The proof is completed.
\end{proof}

 \begin{lemma}\label{lma3}
There are no non-trivial $\frac 12$-derivations of the Schr\"{o}dinger algebra $\mathcal{S}_{n}$ for $n \neq 2$.
\end{lemma}

\begin{proof}  The result for $n=1$ is given by \cite{fml}. So below we assume that $n\geq 3$.

Let $\varphi$ be a $\frac 12$-derivation of $\mathcal{S}_{n}$. As $\mathcal{S}_{n}$ is a finitely generated $\mathbb{Z}_{2}$-graded Lie algebra, by Lemma \ref{l01}, we know that
$$\varphi=\varphi_0+ \varphi_1$$ which satisfying
$$
\varphi_{0}((\mathcal{S}_{n})_{\bar{0}})\subseteq(\mathcal{S}_{n})_{\bar{0}}, \quad \varphi_{0}((\mathcal{S}_{n})_{\bar{1}})\subseteq(\mathcal{S}_{n})_{\bar{1}}, \quad
\varphi_{1}((\mathcal{S}_{n})_{\bar{0}})\subseteq(\mathcal{S}_{n})_{\bar{1}}, \quad \varphi_{1}((\mathcal{S}_{n})_{\bar{1}})\subseteq(\mathcal{S}_{n})_{\bar{0}},
$$
where $\varphi_0, \varphi_1$ are both graded $\frac 12$-derivations  of $\mathcal{S}_{n}$ and $(\mathcal{S}_{n})_{\bar{0}}, (\mathcal{S}_{n})_{\bar{1}}$ are given by (\ref{Grad}).
From Lemmas \ref{12} and \ref{13} we see that $\varphi_0$ is trivial and $\varphi_1=0$ respectively. Therefore, one has $\varphi=\varphi_0+ \varphi_1=\varphi_0$, i.e., $\varphi$ is trivial. The proof is completed.
\end{proof}

For $n=2$, one can apply a similar method to that used to prove Lemmas \ref{12},\ref{13} and \ref{lma3},
where only the formula (\ref{sform2}) cannot be obtained because $n=2$ so that $s_{ij}, s_{ik}$ cannot be found from $s_{kj}$, and then we obtain the following lemma.

\begin{lemma} \label{lmtang}
$\varphi$ is a $\frac 1 2$-derivation of the Schr\"{o}dinger algebra $S_2$ if and only if there are $\theta, \beta\in \mathbb{C}$ such that
 \begin{equation}\label{dern=2}
 \begin{array}{llllll}
 \varphi(e) = \theta e, & \varphi(f) = \theta f, & \varphi(h) = \theta h, \\[1mm]
\varphi(x_i) = \theta x_i, & \varphi(y_i) = \theta y_i, &  1 \leq i \leq 2, \\[1mm]
 \varphi(z) = \theta z,& \varphi(s_{12}) = \theta s_{12} + \beta z.
 \end{array}
 \end{equation}
 That is, every $\frac 1 2$-derivation $\varphi$ of the Schr\"{o}dinger algebra $S_2$  is of the form $$\varphi=\theta {\rm{id}}_{S_2} +\beta \Re$$ for some $\theta, \beta\in \mathbb{C}$, where $\Re$ is given by (\ref{ders2}).
 \end{lemma}

  By Lemmas \ref{lma3} and \ref{lmtang} one has the following theorem on $\frac 1 2$-derivation of  $S_n$.

 \begin{theorem} \label{th}
Every $\frac 1 2$-derivation $\varphi$ of the Schr\"{o}dinger algebra $S_n$  is of the form $$\varphi=\theta {\rm{id}}_{S_n} +\delta_{2n}\beta \Re$$ for some $\theta, \beta\in \mathbb{C}$, where $\Re$ is given by (\ref{ders2}).  In other words, we have
$$\triangle (S_n)={\rm Span}_{\mathbb{C}}\{{\rm id}_{S_n}, \delta_{2n}\Re\},$$
which implies that the dimensions of the space $\triangle (S_n)$ are
$\dim \triangle (S_n)=1$ if $n\neq 2$ and  $\dim \triangle (S_2)=2$.
 \end{theorem}

 We now give the following main result in this paper.
\begin{theorem}
If $n\neq 2$, then all the transposed Poisson algebra structures on $(\mathcal{S}_{n}, [\cdot, \cdot])$ are trivial. If $n=2$, then, up to an isomorphism,
there is only one non-trivial transposed Poisson structure $\cdot$ on $(S_2,[\cdot,\cdot])$ given by
$$ s_{12}\cdot s_{12}= z.$$
\end{theorem}

\begin{proof}
By Lemmas \ref{l1}, \ref{t1} and Theorem \ref{th}, the statement follows.
\end{proof}

Let us recall the definition of ${\rm Hom}$-structures on Lie algebras.
	\begin{definition}
		Let $({\mathcal {L}}, [\cdot,\cdot])$ be a Lie algebra and $\varphi$ be a linear map.
		Then $({\mathcal {L}}, [\cdot,\cdot], \varphi)$ is a ${\rm Hom}$-Lie structure on $({\mathcal {L}}, [\cdot,\cdot])$ if
		\[
		[\varphi(x),[y,z]]+[\varphi(y),[z,y]]+[\varphi(z),[x,y]]=0.
		\]
	\end{definition}

Filippov proved that each nonzero $\delta$-derivation ($\delta\neq 0,1$) of a Lie algebra,
 gives a non-trivial ${\rm Hom}$-Lie algebra structure \cite[Theorem 1]{fil1}.
 Hence, by \eqref{dern=2}, we have the following corollary.

 \begin{corollary}
The Schr\"{o}dinger algebra $\mathcal{S}_{2}$   admits a non-trivial ${\rm Hom}$-Lie algebra structure.
 \end{corollary}

\section*{Funding}

This work is supported in part by NSF of China (No. 12271085), NSF of Heilongjiang Province (No. LH2020A020) and the fund of Heilongjiang Provincial Laboratory of the Theory and Computation of Complex Systems.

\section*{Data availability}

No data was used for the research described in the article.

\end{document}